\documentclass[12pt]{amsart}

\title{On the Prague dimension of sparse random graphs}

\author{Felix Joos}
\author{Letícia Mattos}

\address{Institut für Informatik, Universität Heidelberg, Im Neuenheimer Feld 205, 69120 Heidelberg, Germany.}\email{joos|mattos@uni-heidelberg.de}

\date{}

\thanks{The research leading to these results was partially supported by the Deutsche Forschungsgemeinschaft (DFG, German Research Foundation) -- 428212407.
}

\usepackage{geometry}
\usepackage{mathrsfs}
\usepackage{bbm}
\usepackage{dsfont}
\usepackage{amsmath}
\usepackage{indentfirst}
\usepackage{mathtools}
\usepackage{amsfonts}
\usepackage{amsthm}
\usepackage[english]{babel}
\usepackage{enumerate}
\usepackage{color}
\usepackage{theoremref}
\usepackage{enumitem}

\usepackage{ amssymb }
\usepackage{caption}
\usepackage[breaklinks]{hyperref}
\hypersetup{
	colorlinks = true, 
	linkcolor = blue, 
	citecolor = blue 
}

\newtheorem{theorem}{Theorem}

\newtheorem{lemma}[theorem]{Lemma}

\numberwithin{theorem}{section}
\theoremstyle{remark}

\def\E{\mathcal{E}}

\def\F{\mathcal{F}}

\def\se{\subseteq}
\def\td{\widetilde{d}}
\def\pdim{\mathrm{pdim}}
\def\idim{\mathrm{idim}}

\newcommand{\eps}{\varepsilon}
\renewcommand{\Pr}[1]{\mathbb{P}\left[#1\right]}
\newcommand{\Ex}[1]{\mathbb{E}\left[#1\right]}

\definecolor{lblue}{rgb}{0.5,0.5,1}

\begin{document}

\vspace{-0.5cm}

\begin{abstract}

	The Prague dimension of a graph $G$ is defined as the minimum number of complete graphs whose direct product contains $G$ as an induced subgraph. 
	Introduced in the 1970s by Nešetřil, Pultr, and Rödl -- and motivated by the work of Dushnik and Miller, as well as by the induced Ramsey theorem -- determining the Prague dimension of a graph is a notoriously hard problem.
	In this paper, we show that for all $\eps > 0$ and $p$ such that $ n^{-1+\eps} \le p \le n^{-\eps}$, with high probability the Prague dimension of $G_{n,p}$ is $\Theta_{\eps}(pn)$, which improves upon a recent result by Molnar, Rödl, Sales and Schacht.

	Inspired by the work of Bennett and Bohman, our approach centres on analysing a random greedy process that builds an independent set of size $\Omega(p^{-1}\log pn)$ by iteratively selecting vertices uniformly at random from the common non-neighbourhood of those already chosen.
	Using the differential equation method, we show that every non-edge is essentially equally likely to be covered by this process, which is key to establishing our bound.

\end{abstract}

\maketitle

\section{Introduction}

Motivated by the work of Dushnik and Miller, as well as by the induced Ramsey theorem~\cite{Deuber,ErdosHajnalPosa,Rodl},
Nešetřil, Pultr and Rödl~\cite{Nesetril1977-pj,Nesetril1978-km} introduced the notion of the Prague dimension of a graph in the 1970s.
For a graph $G$, the Prague dimension $\text{pdim}(G)$ is defined as the minimum number of complete graphs whose direct product contains $G$ as an induced subgraph.
This is closely related to $\text{idim}(G)$, defined as the minimum number of partial clique factors whose union equals $\overline{G}$, the complement of $G$.
In fact, 
Nešetřil and Rödl~\cite{Nesetril1978-km} showed that for every graph $G$, we have
\begin{align}\label{eq:pdim-idim}
	\idim(G) \le \pdim(G) \le \idim(G) + 1. 
\end{align}

In 1977, Lovász, Nešetřil, and Pultr~\cite{Lovasz1980-kw} showed that a graph on $n\ge2$ vertices has Prague dimension at most $n-1$ and classified all graphs which achieve this bound.
General upper and lower bounds were obtained by Alon~\cite{Alon1986-oc}, who proved that if $G$ is a graph on $n$ vertices with minimum degree at least 1 and maximum degree $\Delta$, then $\log (n/ \Delta)/ \log 2 \le \text{pdim}(G) = O(\Delta^2\log n)$; all our asymptotic expressions are meant as $n \to \infty$, and all logarithms are natural.
The upper bound was later improved by Eaton and Rödl~\cite{Eaton1996-ap}, who showed that 
\begin{align*}
	\text{pdim}(G) = O(\Delta \log n).
\end{align*}
Moreover, they proved that this is best possible up to a factor of $1/(\log \Delta + \log \log (n/2\Delta) )$.
When $\Delta$ is a constant, these results provide the order of magnitude of $\text{pdim}(G)$.
Precise asymptotics are known for matchings, trees and cycles~\cite{Alon1986-oc,Lovasz1980-kw,Poljak1981-lt}, hypercubes~\cite{Krivka} and Kneser graphs $K(n,k)$ for constant values of $k$~\cite{Poljak1978-um}.
For more details, we refer the reader to~\cite{Kantor}.

Nešetřil and Rödl~\cite{Nesetril1983-ju} pioneered the study of the Prague dimension of random graphs by showing that the Prague dimension of $G_{n,p}$ is $\Omega_p(n/\log n)$ with high probability for constant values of $p$.
Many years later, Guo, Patton and Warnke~\cite{Guo2023-wq} managed to obtain a matching upper bound by using a Pippenger--Spencer~\cite{Pippenger1989-ak} type
edge-colouring result for random hypergraphs with large uniformities.
More recently, Molnar, Rödl, Sales and Schacht~\cite{MolnarRodlSalesSchacht} studied the Prague dimension of sparse random graphs.
Via a double counting argument, they showed that for $p \gg n^{-2}$, we have\footnote{For any two functions $f$ and $g$, we write $f \ll g$ to mean $f = o(g)$.}
\begin{align}\label{eq:lower-bound-MRSS}
	\text{pdim}(G_{n,p}) \ge \dfrac{pn \log \frac{1}{p} }{5\log n }
\end{align}
with high probability.
Furthermore, by analysing the distribution of independent sets in~$G_{n,p}$, they showed that for all $\eps > 0$ and $p$ such that $n^{-1/3} \log^{4/3} n \ll p \ll n^{-\eps}$, we have $\text{pdim}(G_{n,p}) = \Theta_{\eps}(pn)$ with high probability.
In our first main result, we prove an upper bound that is valid for almost the entire range of $p$.

\begin{theorem}\label{thm:main-prague}
	Let\footnote{In all our statements, $\eps$ is a fixed constant and $n$ tends to infinity.} $\eps \in (0,1)$ and $(\log n)^{2+\eps} n^{-1} \le p < n^{-\eps}$. 
	Then, with high probability we have
	\begin{align*}
		\mathrm{pdim}(G_{n,p}) = O_{\eps} \left ( \dfrac{pn \log n}{\log pn} \right ).
	\end{align*}
\end{theorem}

In particular, by combining our result with the lower bound in~\eqref{eq:lower-bound-MRSS}, we obtain that for all $\eps > 0$ and $p$ such that $ n^{-1+\eps} \le p \le n^{-\eps}$, with high probability the Prague dimension of $G_{n,p}$ is $\Theta_{\eps}(pn)$.
As it is crucial for our analysis to have concentration on the degrees of all vertices in $G_{n,p}$, we note that $n^{-1}\log n$ is a natural barrier for the range of $p$ our method can handle.

To prove Theorem~\ref{thm:main-prague}, instead of working directly with the Prague dimension, we use~\eqref{eq:pdim-idim} and upper bound $\idim(G_{n,p})$.
Our main idea is to select independent sets with size proportional to $\alpha(G_{n,p})$ via a random greedy process which takes one vertex at a time uniformly at random from the common non-neighbourhood of the vertices already chosen.
We then control the probability that a given vertex and a given non-edge are covered by an independent set generated by our process.

Our strategy is inspired by the work of Bennett and Bohman~\cite{Bennett2016-di}, who use a random greedy independent set process to lower bound the independence number of regular hypergraphs under certain degree and codegree conditions.
For graphs, we also improve on their work by handling pseudorandom graphs with logarithmic degrees rather than regular graphs with degrees growing polynomially with the number of vertices.
Our analysis is done via the differential equation method, which was popularised in the combinatorial community by Wormald~\cite{Wormald-1,Wormald-2}.
As can be seen in the references given in~\cite{Wormald-2}, the roots of the method may be traced back further, see for example the works of Kurtz~\cite{Kurtz1970-aa}
and Karp and Sipser~\cite{Karp1981-aa}.

Our methods also allow us to bound the minimum number of independent sets needed to cover all the non-edges of $G=G_{n,p}$, denoted by $\overline{\theta}_1(G)$.
Initially, one may conjecture that with high probability $\overline{\theta}_1(G)$ should be of order $n^2/\alpha(G)^2$  for all values of $p$, which corresponds to the size of a `near' optimal clique cover of $\overline{G}$.
Frieze and Reed~\cite{Frieze1995-xa}, by improving the results of Bollobás, Erd\H{o}s, Spencer, and West~\cite{Bollobas1993-nl}, showed that this is indeed the case for constant values of $p$.
Surprisingly, this is not true for the sparse range.
In~\cite{Guo2023-wq}, Guo, Patton and Warnke proved that for any fixed $\eps >0$ and $n^{-1} \ll p \le 1-\eps$, with high probability we have
\begin{align*}
	\overline{\theta}_1(G_{n,p}) = \Omega \left ( \dfrac{n^2 \log \frac{1}{p} }{\alpha^2(G_{n,p})} \right ).
\end{align*}
It is worth mentioning that due to the work of Matula~\cite{Matula1970-ka}, Bollobás and Erd\H{o}s~\cite{Bollobas1976-al}, Grimmett and McDiarmid~\cite{Grimmett1975-vs}, and Frieze~\cite{Frieze1990-cu}, we know that the size of the largest independent set in $G_{n,p}$ is of order $\Theta(p^{-1}\log pn)$ for $n^{-1} \ll p \le 1- \eps$, for any fixed $\eps > 0$.
Also note that
$\log \frac{1}{p} = \Theta_{\eps}(\log n)$ for $ n^{-1}\ll p \ll n^{-\eps}$. 

Recently, Molnar, Rödl, Sales and Schacht~\cite{MolnarRodlSalesSchacht} obtained a matching upper bound for every $\eps > 0$ and $p$ such that $n^{-1/2} \log^{5/4} n \ll p \ll n^{-\eps}$. 
Their argument heavily relies on the fact that the number of independent sets in $G_{n,p}$ is concentrated around its expectation with high probability.
As this no longer holds for $p \ll n^{-1/2}$ (see Corollary~19 in~\cite{Coja-Efthymiou}), there is a barrier at $n^{-1/2}$ for the range of $p$ their method can handle.
Our second main theorem overcomes this natural barrier and extends their result to a much wide range of~$p$.

\begin{theorem}\label{thm:main}
	Let $\eps\in (0,1)$ and $(\log n)^{2+\eps} n^{-1} \le p < n^{-\eps}$.
	Then, with high probability we have
	\begin{align*}
		\overline{\theta}_1(G_{n,p}) = \Theta_{\eps} \left ( \dfrac{n^2\log n}{\alpha^2(G_{n,p})} \right ).
	\end{align*}
\end{theorem}

It is worth noting we were only able to overcome the previous obstacles because for many host graphs the independent set generated by our random greedy process is not uniformly distributed.
As an example, let $1 \ll k \ll n$ and take a complete bipartite graph $K[A,B]$ with $|A| = n/3$ and $|B| = 2n/3$.
Now let $I$ be a uniformly-chosen random independent $k$-set and let $J$ be the independent set generated by our random greedy process.
One can check that if $u, v \in A$, then $\Pr{u,v \in I} \ll \Pr{u,v \in J}$.

Now let us mention some of the historic background related to $\overline{\theta}_1$.
In the 1960s, Erd\H{o}s, Goodman and Pósa~\cite{Erdos1966-so} introduced the parameter $\theta_1(G)$, defined as the minimum number of cliques required to cover all the edges of $G$.
Notice that $\overline{\theta}_1(G) = \theta_1(\overline{G})$, and hence one may think of $\theta_1(G)$ and $\overline{\theta}_1(G)$ as the dual versions of each other.
In~\cite{Erdos1966-so}, the authors showed that $\theta_1(G) \le \lfloor n^2/4 \rfloor$ for every $n$-vertex graph $G$, which is attained by balanced complete bipartite graphs.
In 1994, Alon~\cite{Alon1986-oc} showed that if $G$ is a $n$-vertex graph with maximum degree $\Delta$, then $\theta_1(\overline{G}) = O(\Delta^2 \log n)$.
Eaton and Rödl~\cite{Eaton1996-ap} showed that this bound is close to best possible, as there are $n$-vertex graphs $G$ with $\theta_1(\overline{G}) = \Omega(\Delta^2 \log n / \log \Delta)$.

Frieze and Reed~\cite{Frieze1995-xa} showed that with high probability $\theta_1(G_{n,p}) = \Theta_p(n^2/\log n)$ for constant values of $p$.
Guo, Patton and Warnke~\cite{Guo2023-wq} extended these results much further by showing that for every fixed $\gamma \in (0,1)$ and $p$ such that $n^{-2} \le p \le 1-\gamma$, with high probability, we have $\theta_1(G_{n,p}) = \Theta (n^2p /\log_{1/p}^{2} n )$.
As results in~\cite{Guo2023-wq} left the range where $p$ tends to 1 open, it is more convenient to just consider $\overline{\theta}_1(G_{n,p})$ for $p$ tending to 0.

The rest of the paper is divided as follows.
In Section \ref{sec:proofs}, we present the proofs of our main theorems; in Section~\ref{sec:random-greedy}, we analyse our random greedy independent set process using the differential equation method; in Section~\ref{sec:prob-cover}, we prove the main technical result of the paper, namely Theorem~\ref{thm:main-dem}; and in Section \ref{sec:preliminaries}, we present some tools and standard lemmas that will be used in the analysis of our random process.

\section{Proofs of the main theorems}\label{sec:proofs}

We consider the following random process, which we call the \emph{random greedy independent set process}.
Given a graph $G$, we select vertices at random from its vertex set as follows.
We choose $v_1$ uniformly at random in $V(G)$.
For $i \ge 2$, we choose $v_i$ uniformly at random from the common non-neighbourhood of $v_1,\ldots,v_{i-1}$ in $G$, which is defined as
\begin{align}\label{eq:non-neighbour}
	N^c_G(\{v_1,\ldots,v_{i-1}\}) \coloneqq V(G) \setminus \bigcup_{j=1}^{i-1}  \big (\{v_j\} \cup N_G(v_j)\big ) .
\end{align}
If such a vertex does not exist, we stop the process.
For each step $i$ before stopping, we denote by $I_i(G)$ the random independent set $\{v_1,\ldots,v_i\}$ built by this process.

Here we are mainly interested in analysing this random process for \emph{$(\eps,p)$-typical} graphs.
To this end, for $p \in (0,1)$ and $n,i$ positive integers, we set
\begin{align*}
	f_0(n,p) \coloneqq 4\log n \sqrt{ \dfrac{\log pn}{pn} + p } \qquad \text{and} \qquad f_i(n,p) :=  \left ( \dfrac{1+16p}{1-p} \right )^i f_0(n,p).
\end{align*}
We write $f_i(n,p) = f_i$ whenever the parameters $n$ and $p$ are clear from the context.
We choose the growth rate of $f_i$ so that, up to an appropriate stopping time, the process $\big(d_i(v) - (1+f_i)(1-p)^i p n\big)_i$ is a supermartingale for all $v$; see Lemma~\ref{lem:supermartingale} and~\eqref{eq:increment-fidi}.
The initial value $f_0$ is determined by an application of Freedman's inequality in this analysis; see Theorem~\ref{thm:main-dem} and~\eqref{eq:freedman-appl-1}.

We say that a graph $G$ on $n$ vertices is $(\eps,p)$-typical if it satisfies the following three properties:
\begin{enumerate}[label=(P\arabic*)]
	\item \label{item:P1} For every set $S \se V(G)$ of size at most $\eps p^{-1} \log pn$, we have\footnote{We use the notation $x = a \pm b$ for $x \in [a-b,a+b]$.}
	\begin{align*}
		|N^c_G(S)| = \left ( 1 \pm f_{|S|} \right ) (1-p)^{|S|} n .
	\end{align*}
	
	\item \label{item:P2} For every vertex $v \in V(G)$, we have 
	\begin{align*}
		d_G(v) = \left (1 \pm \dfrac{f_0}{2} \right ) pn .
	\end{align*}
	\item \label{item:P3} For all pairs of distinct vertices $u,v \in V(G)$, we have
	\begin{align*}
		|N_G(u) \cap N_G(v)| \le \Delta_2 \coloneqq 4p^2n + 2^7 \log n.
	\end{align*}
\end{enumerate}
We omit $G$ from the notation when it is clear from the context.

Our next lemma states that, for the range of $p$ we are interested in, $G_{n,p}$ is $(\eps,p)$-typical with high probability.
\begin{lemma}\label{lem:typical}
	Let $\eps \in (0,2^{-10})$ and $(\log n)^{2+\eps} n^{-1} \le p < n^{-\eps}$. Then, $G_{n,p}$ is $(\eps,p)$-typical with probability at least $1-O(n^{-3})$.
\end{lemma}

The proof of Lemma~\ref{lem:typical} is somewhat standard, and hence we postpone it to the end of the paper (Section~\ref{sec:preliminaries}).
Our next theorem gives the asymptotic value for the probability that a given vertex and a given non-edge of an $(\eps,p)$-typical graph $G$ is covered by $I_k(G)$, that is, the independent set generated by our random greedy process, for $p$ in the range of interest and $k = \eps 2^{-10}p^{-1} \log pn$ (throughout the article, we omit roundings if they do not affect the argument).
This is a key ingredient in the proof of both Theorems~\ref{thm:main-prague} and~\ref{thm:main}.

\begin{theorem}\label{thm:prob-non-edge}
	Let $\eps \in (0,2^{-10})$ and $(\log n)^{2+\eps} n^{-1} \le p < n^{-\eps}$. 
	Let $G$ be an $n$-vertex $(\eps,p)$-typical graph and $k = \eps 2^{-10} p^{-1} \log pn $.
	Then, for all distinct $u, v \in V(G)$ such that $uv \notin E(G)$, we have
	\begin{align*}
		\Pr{v \in I_k} = \big(1+o_{\eps}(1) \big ) \dfrac{k}{n}  \qquad 
		\text{and} \qquad \Pr{u,v \in I_k} = \big  (1+o_{\eps}(1) \big )  \left ( \dfrac{k}{n} \right )^2.
	\end{align*}
\end{theorem}

In Theorem~\ref{thm:prob-non-edge}, showing that $\Pr{v \in I_k} = \Theta_{\eps} (k/n)$ is essentially equivalent to proving that for every $i \in [k]\coloneqq \{1,\ldots,k\}$ we have $\Pr{v_i = v} = \Theta_{\eps} (1/n)$ (similarly for the probability that $u,v \in I_k$).
For the event $\{ v_i = v \}$ to occur, it is necessary that none of the vertices $v_1,\ldots,v_{i-1}$ is adjacent to $v$.
By invoking property (P1) of $(\eps,p)$-typical graphs, one can show that for each $j$, the probability that $v_j$ is not a neighbour of $v$ given that $\{ v_1,\ldots,v_{j-1} \} \se N^c(v)$, is of order $(1\pm 2f_j)(1-p)$.
Unfortunately, this rough estimate is insufficient to deduce that  $\Pr{v_i = v} = \Theta_{\eps} (1/n)$ for all $i \in [k]$, since the error terms accumulate and the sum $f_1 + \ldots + f_k = e^{\Theta(pk)}p^{-1}f_0$ does not tend to zero as $n$ grows.

In order to overcome this difficulty, we use the differential equation method to have a refined control on the degrees of all vertices in the subgraph $G_j$ induced by $V_j = N^{c}(v_1,\ldots,v_j)$ for every $j \in [k]$. 
Roughly speaking, in Theorem~\ref{thm:main-dem} we show that with high probability for every $j \in [k]$, every vertex in $V_j$ has degree $(1 \pm f_j)p|V_j|$ in $G_j$.
In particular, one can show that for each $j$, the probability that $v_j$ is not a neighbour of $v$ given that $\{ v_1,\ldots,v_{j-1} \} \se N^c(v)$, is of order $1-(1\pm 2f_j)p$.
As $f_1 + \ldots + f_k = o(p^{-1})$, this new estimate plus some careful analysis allows us to derive Theorem~\ref{thm:prob-non-edge}.
See Section~\ref{sec:prob-cover} for more details.

Before proceeding with the proof of Theorem~\ref{thm:main-prague}, we review the necessary background on the Prague dimension, which is closely related to the parameter $\text{idim}$.
For a graph $G$, $\text{idim}(G)$ is the smallest integer $t$ such that there are collections of vertex-disjoint subsets $\mathcal{I}_j = (I_{j,1},\ldots,I_{j,r_j})$ of $V(G)$ for $j \in [t]$ so that 
\begin{align*}
	E(G) = \binom{V(G)}{2} \setminus \bigcup \limits_{j \in [t]} \bigcup \limits_{i \in [r_j]} \binom{I_{j,i}}{2}.
\end{align*}
Here, we denote by $\binom{A}{2}$ the set of all pairs of elements of a set $A$.


\begin{proof}[Proof of Theorem~\ref{thm:main-prague}, assuming Lemma~\ref{lem:typical} and Theorem~\ref{thm:prob-non-edge}]
	Let $G$ be an $(\eps,p)$-typical graph on $n$ vertices.
	We set
	\[ k \coloneqq \eps 2^{-10} p^{-1} \log pn, \qquad s \coloneqq  n k^{-1}  \qquad \text{and} \qquad t \coloneqq C_{\eps} n k^{-1}\log n,\]
	where $C_{\eps} > 0$ is some large constant to be chosen later.
	We aim to use~\eqref{eq:pdim-idim} and, in order to do so, we use many independent instances of the random greedy independent set process described above. As every non-edge is essentially equally likely to be covered, a $\log n$ factor overhead is sufficient to cover all non-edges with high probability.

	Let $\big(J_{i,j}: i \in [t], j \in [s] \big )$ be a collection of $ts$ independent copies of $I_k = I_k(G)$.
	For each $i \in [t]$, define $\mathcal{I}_i$ to be the collection of vertex-disjoint subsets of $V(G)$ given by  
	\begin{align*}
		I_{i,j} \coloneqq J_{i,j} \setminus \big ( J_{i,1} \cup \cdots \cup J_{i,j-1} \big )
	\end{align*}
	for all $j \in [s]$. 
	By~\eqref{eq:pdim-idim}, it suffices to show that with high probability every non-edge of $G$ is covered by some $I_{i,j}$.
	For $i \in [t]$, the probability that a given non-edge $uv$ is covered by an independent set in the partition $\mathcal{I}_i$ is equal to
	\begin{align}\label{eq:prob-non-edge}
		\sum_{j=1}^{s}\Pr{u,v \in I_{i,j}} & = \sum_{j=1}^{s} \Pr{u,v \in J_{i,j}} \prod_{\ell\le j-1} \Pr{ \{ u,v \} \cap J_{i,\ell} = \emptyset}\nonumber\\
		& = \sum_{j=1}^{s} \Pr{u,v \in I_k} \Pr{ \{ u,v \} \cap I_k = \emptyset}^{j-1}.
	\end{align}
	By the inclusion-exclusion principle, the probability that neither $u$ nor $v$ is covered by $I_k$ is 
	\begin{align}\label{eq:prob-non-edge-2}
		\Pr{ \{ u,v \} \cap I_k = \emptyset} &= 1 - \Pr{u \in I_k} - \Pr{v \in I_k} + \Pr{u,v \in I_k} \nonumber \\
		& = 1 - \Theta_{\eps} \left ( \dfrac{k}{n} \right ) = 1 - \Theta_{\eps}(s^{-1}).
	\end{align}
	Here we use Theorem~\ref{thm:prob-non-edge} and the bound $k = o(n)$, which follows from the fact that $p \gg n^{-1}\log n$.

	By combining~\eqref{eq:prob-non-edge},~\eqref{eq:prob-non-edge-2} and Theorem~\ref{thm:prob-non-edge}, we obtain that the probability that a given non-edge $uv$ is covered by some $I_{i,j}$ in $\mathcal{I}_i$ is of order
	\begin{align}\label{eq:prob-non-edge-3}
		\Theta_{\eps} \left ( \dfrac{k^2}{n^2} \right )\sum_{j=1}^{s} \left ( 1 - \Theta_{\eps} \left ( s^{-1} \right )\right )^{j-1} = \Theta_{\eps} \left ( \dfrac{k}{n} \right ) \left ( 1 - \left ( 1 - \Theta_{\eps} \left ( s^{-1} \right ) \right )^s \right ) = \Theta_{\eps} \left ( \dfrac{k}{n} \right ).
	\end{align}
	In the first inequality, we use that $\sum_{j=0}^{s-1} (1-q)^j = q^{-1}(1-(1-q)^s)$ for all $q \in (0,1)$;
	in the last equality, we use that $s =  n/k$, and hence $(1 - (1 - \Theta_{\eps}(s^{-1}))^s) = \Theta_{\eps}(1)$.

	It follows from~\eqref{eq:prob-non-edge-3} that 
	\begin{align}\label{eq:prob-non-edge-4}
		\Pr{ \{u,v\} \not \se I_{i,j} \text{ for all } j \in [t]} = \left ( 1 - \Theta_{\eps}\left ( \dfrac{k}{n} \right ) \right )^t \le n^{-3}.
	\end{align}
	In the last inequality, we choose $C_{\eps}$ large enough so that the corresponding probability is at most $n^{-3}$.

	By the union bound, we conclude that for any $(\eps,p)$-typical graph $G$, the probability that some non-edge of $G$ is not covered by any $I_{i,j}$ with $i \in [t]$ and $j \in [s]$ is at most $n^{-1}$.
	As $G_{n,p}$ is $(\eps,p)$-typical with probability at least $1-O(n^{-3})$ by Lemma~\ref{lem:typical}, this concludes the proof.
\end{proof}

Now we are ready to prove Theorem \ref{thm:main}.

\begin{proof}[Proof of Theorem~\ref{thm:main}, assuming Lemma~\ref{lem:typical} and Theorem~\ref{thm:prob-non-edge}]
	Since $\alpha(G_{n,p})$ is of order $p^{-1}\log pn$ (see~\cite{Bollobas1976-al,Frieze1990-cu,Grimmett1975-vs,Matula1970-ka}), it suffices to show that with high probability we can find a collection of $O_{\eps}(n^2p^2\log n\log^{-2} pn)$ independent sets covering all non-edges of $G_{n,p}$.

	Let $G$ be an $(\eps,p)$-typical graph on $n$ vertices and set
	$k =  \eps 2^{-10} p^{-1} \log pn $.
	By Theorem~\ref{thm:prob-non-edge}, 
	for all $uv \notin E(G_{n,p})$, we have $\Pr{u,v \in I_k} \ge k^2n^{-2}/2$,
	where $I_k=I_k(G)$ is the random independent as defined at the beginning of Section~\ref{sec:proofs}.

	Set $t = 6k^{-2} n^2\log n $ and let $I^{(1)}_k,\ldots, I^{(t)}_k$ be independent copies of $I_k$.
	We have
	\begin{align*}
		\Pr{\{u,v\} \not \se I_k^{(j)} \text{ for all } j \in [t]} = \prod_{j=1}^t \Pr{\{u,v\} \not \se I_k^{(j)}} \le \left ( 1 - \dfrac{k^2}{2n^2} \right )^t \le e^{-3\log n}
	\end{align*}
	for every non-edge $uv$ of an $(\eps,p)$-typical graph $G$.
	As $G_{n,p}$ is $(\eps,p)$-typical with probability at least $1-O(n^{-3})$ by Lemma~\ref{lem:typical}, this together with the union bound over all non-edges of $G_{n,p}$ 
	shows that with high probability at most $t$ independent sets are needed to cover all non-edges of $G_{n,p}$.
\end{proof}

\section{The random greedy independent set process}~\label{sec:random-greedy}

In this section, we analyse the random greedy independent set process described in Section~\ref{sec:proofs} and show that, with high probability, the degrees of the vertices which are active at step $i$ are close to their expected trajectories.

Let $\eps \in (0,2^{-10})$ be a fixed constant, $n$ be sufficiently large, $(\log n)^{2+\eps} n^{-1} \le p < n^{-\eps}$, $G$ an $(\eps,p)$-typical graph on $n$ vertices and set
\begin{align*}
	k=  \dfrac{\eps\log pn}{2^{10}p}.
\end{align*}
To simplify notation in our analysis, we set $G_0 = G$, $I_0 = \emptyset$ and $V_0 = V(G_0)$.
For $i \ge 1$, we set
\[ V_i := N^c(\{v_1,\ldots,v_{i}\}), \qquad I_i := \{v_1,\ldots,v_{i}\} \qquad \text{and} \qquad G_i = G[V_i].\]
Note that $I_i$ is the independent set built until step $i$, $V_i$ is the set of available (or active) vertices at step $i$ and $G_i$ is the subgraph of $G_0$ induced by $V_i$.
In addition, $V_i\cap I_i = \emptyset$ for all $i$.

For $i \ge 0$ and $u,v \in V_i$, we denote by $N_i(v)$ the neighbourhood of $v$ in $G_i$ and we set
\[d_i(v) = |N_i(v)| \qquad \text{and} \qquad d_i(u,v) = \big | N_i(u) \cap N_i(v) \big |.\]
We would like to track the evolution of $d_i(v)$ for every vertex in $V_0$ and show that it is well-behaved.
For this to be possible, the first step is to define an expected trajectory. 
For $i \ge 0$, we set 
\begin{align}\label{eq:expected-trajectory-degree}
	\td_i := (1-p)^ipn.
\end{align}

Let us also recall the definition of the error terms that we use to control the deviation of the actual degrees from the expected trajectory. 
For $i \ge 0$, we set
\begin{align}\label{eq:error-function}
	f_0 \coloneqq 4\log n \sqrt{ \dfrac{\log pn}{pn} + p } \qquad \text{and} \qquad f_i :=  \left ( \dfrac{1+16p}{1-p} \right )^i f_0.
\end{align}
Note that the error is increasing, yet we have 
\begin{align}\label{eq:bound-f}
	f_{i-1} < f_{i} \le f_k \le (pn)^{2^{-5}\eps} f_0 = o(1) \qquad \text{and} \qquad
	\dfrac{1}{1\pm f_i} = 1 \pm \dfrac{3}{2}f_i
\end{align}
for all $i \in [k]$ and  $(\log n)^{2+\eps} n^{-1} \le p < n^{-\eps}$.
To derive the inequality $f_k \le (pn)^{2^{-5}\eps}f_0$, we simply use that $(1+16p)(1-p)^{-1} \le 1+2^5p \le e^{2^5p}$.
It follows from~\ref{item:P1} that
\begin{align}\label{eq:initial-parameters-2}
	|V_i| = (1\pm f_i)(1-p)^i n .
\end{align}
for all $i \in [k]$. 
Other useful estimates are
\begin{align}\label{eq:asymptotics}
		(1-3p)^k \ge e^{-8pk} \ge (np)^{-2^{-7}\eps} \ge n^{-2^{-7}(\eps - \eps^2)} \ge n^{-2^{-7}\eps } \quad
		\text{and} \quad f_i > f_0 \ge n^{-1/2}
	\end{align}
for all $i \in [k]$. 
Above we use that $1-x \ge e^{-2x}$ for all $x \in (0,1/2)$.

We also recall that $\Delta_2 = 4 p^2n \, + \, 2^7 \log n$ is the error function used to control the size of the common neighbourhood of any two distinct vertices in $V_0$, see~\ref{item:P3}.
By our choice of $p$ and $f_i$, note that
	\begin{align*}
		\Delta_2 = O(1)\max\{p^2n,\log n\} \ll 4pn\log n \sqrt{ \dfrac{\log pn}{pn} + p } = f_0\td_0.
	\end{align*}
To see this, note that $p < p^{1/2}$ and $(pn)^{-1} \ll (pn)^{-1/2}$.
Thus,
\begin{align}\label{eq:bound-Delta-2}
	1 \le \Delta_2 \ll f_0 \td_0 < f_i \td_i
\end{align}
for all $i \ge 1$. The estimate
\begin{align}\label{eq:for-closed-neighbourhood}
	(1\pm f_i)\td_i \pm 1 = \left ( 1 \pm \dfrac{4}{3} f_i \right ) \td_i
\end{align}
for all $i \ge 0$ shall also be useful in our analysis.

Our ultimate goal in this section is to show the following theorem.

\begin{theorem}\label{thm:main-dem}
	Let $\eps \in (0,2^{-10})$ and $(\log n)^{2+\eps} n^{-1} \le p < n^{-\eps}$. 
	Let $G$ be an $n$-vertex $(\eps,p)$-typical graph and $k = \eps 2^{-10}p^{-1} \log pn$.
	Then, the following holds for all sufficiently large $n$. With probability at least $1- n^{-2^{-11}\log pn}$, in the random greedy independent set process, we have
	\[ d_i(v) = (1 \pm f_i)\td_i\] 
	for all $v \in V_i$ and all $i \in [k-1]$.
\end{theorem}

If $v \in V_i$, then $d_i(v)= |N(v) \cap N^c(\{ v_1,\ldots,v_i\})|$. When $G=G_{n,p}$ and the vertices $v$ and $v_1,\ldots,v_i$ are fixed, $d_i(v)$ is a binomial random variable with parameters $n-i-1$ and $(1-p)^ip$.
Thus, one may hope to combine Chernoff's inequality with the union bound to prove Theorem~\ref{thm:main-dem} in the case where $G=G_{n,p}$.
However, by doing so, one would obtain an upper bound which decays at least as $\exp(-pn+i\log n)$ for the probability that $d_i(v)\neq (1\pm f_i)\td_i$ for some $v$.
For $i = k$, this tends to 0 only if $p \ge n^{-1/2+o(1)}$.

The random greedy independent set process helps to overcome the limitations of the union bound when $p$ is significantly smaller than $n^{-1/2}$. A `typical' selection of $v_1, \ldots, v_i$ allows us to use the differential equation method to show that, with high probability, the degrees of vertices in $V_i$ remain close to their expected values.
Notably, $n^{-1}\log n$ emerges as a natural threshold for the range of $p$ where our method applies. This threshold ensures that the degrees in the host graph, such as $G_{n,p}$, remain close to the average degree -- an essential condition for our approach.

Let $\tau$ be the minimum between $k$ and the first time $i$ that $d_i(v) \neq (1 \pm f_i)\td_i $ for some $v \in V_i$.
For a vertex $v \in V_0$, we set
\begin{align}
	\sigma_v = \min \{ i: v \notin V_i \} \qquad \text{and} \qquad \rho_v = \min \{ \tau, \sigma_v-1 \}.
\end{align}
For each $v \in V_0$, we define the random processes $(X_{v,i}^{-})_{i \ge 0}$ and $(X_{v,i}^{+})_{i \ge 0}$ as follows: for ${\boldsymbol{\oplus}} \in \{-1,1\}$, let
\begin{align*}
	X_{v,i}^{{\boldsymbol{\oplus}}} = \begin{cases*} d_i(v) - \td_i \, {\boldsymbol{\oplus}} \, f_i \td_i, \text{
		 if } i \le \rho_v \\  d_{\rho_v}(v) - \td_{\rho_v} \, {\boldsymbol{\oplus}} \, f_{\rho_v} \td_{\rho_v}, \text{ if } i > \rho_v. \end{cases*}
\end{align*}
We set $(\mathcal{F}_i)_{i \ge 0}$ to be its natural filtration, that is, $\mathcal{F}_i$ is the $\sigma$-algebra generated by the random variables $X_0,\ldots,X_i$.

For every sequence $(h_i)_{i\ge 0}$ and every $i \ge 1$, define $\Delta h_i = h_{i} - h_{i-1}$ to be the increment of $h$ (sometimes also called discrete derivative). 
Our next lemma shows that the random processes $(X_{v,i}^{-})_{i \ge 0}$ and $(-X_{v,i}^{+})_{i \ge 0}$ are supermartingales. We also provide upper bounds for the absolute value of the increment and the expected increment of both processes.
Recall that a discrete random process $(Z_i, \mathcal{F}_i)_{i \ge 0}$ in a probability space $(\Omega,\mathcal{F},\mathbb{P})$ is a supermartingale if $\Ex{Z_i \mid \mathcal{F}_{i-1}} \le Z_{i-1}$ for every $i \ge 1$. 

\begin{lemma}\label{lem:supermartingale}
	For every $i\ge 1$ and $v \in V_0$, we have
	\begin{enumerate}[label=\normalfont(\roman*)]
		\item $\Ex{\Delta X_{v,i}^{-} | \F_{i-1} } \le 0$ and $\Ex{\Delta X_{v,i}^{+} | \F_{i-1} }\ge 0$;
		\vspace*{2mm}
		\item $ \Ex{|\Delta X_{v,i}^{\boldsymbol{\oplus}}| \big | \F_{i-1} } \le 3p\td_{i-1}$;
		\vspace*{2mm}
		\item $\max_i |\Delta X_{v,i}^{\boldsymbol{\oplus}}| \le 6p^2n + 2^7\log n$.
	\end{enumerate}
\end{lemma}

\begin{proof}
	We have $\Delta (X^{\boldsymbol{\oplus}}_{v,i}) = 0$ if $i \ge \rho_v+1 = \min\{\tau,\sigma_v-1\}+1$.
	In other words, if either $i \ge \tau +1$ or $i \ge \sigma_v$, then $\Delta (X^{{\boldsymbol{\oplus}}}_{v,i}) = 0$. Therefore,
	\[ \Delta (X^{{\boldsymbol{\oplus}}}_{v,i})  = \Delta (X^{{\boldsymbol{\oplus}}}_{v,i}) 1_{\{ \tau \ge i, \, \sigma_v > i\}} 
	\quad \text{and} \quad 
	\Ex{\Delta X_{v,i}^{{\boldsymbol{\oplus}}} | \F_{i-1} }  = \Ex{\Delta (X_{v,i}^{{\boldsymbol{\oplus}}}) 1_{ \{ \sigma_v > i\}}| \F_{i-1} } 1_{\{ \tau \ge i \}}.\]
	In the second equality we use that the event $\{\tau \ge i \}$ is $\mathcal{F}_{i-1}$-measurable.
	
	The event $\{ \sigma_v > i \}$ occurs if and only if $v \in V_i$.
	Moreover, $v \in V_i$ if and only if $v \in V_{i-1}$ and $v_i \notin N_{i-1}[v]:= N_{i-1}(v) \cup \{v\}$.
	Thus, we can write
	\begin{align}\label{eq:increment}
		\Delta (X^{{\boldsymbol{\oplus}}}_{v,i}) 1_{\{\sigma_v > i\}} & = \big ( - d_{i-1}(v_i, v) - \Delta \td_i \, {\boldsymbol{\oplus}} \,\Delta (f_i \td_i)\big )1_{\{\sigma_v > i\}} \nonumber \\
		&= \sum \limits_{u \in V_{i-1} \setminus N_{i-1}[v]} \Big (- d_{i-1}(u,v) - \Delta \td_i \, {\boldsymbol{\oplus}} \,\Delta (f_i \td_i) \Big )1_{\{ v_i = u\}}1_{ \{ v \in V_{i-1} \}}.
	\end{align}

	For simplicity, for every non-negative $j$ and $v \in V_j$, we set 
	\begin{align*}
		M_{v,j} \coloneqq \sum \limits_{u \in V_{j} \setminus N_{j}[v]} d_{j}(u,v) \qquad \text{and} \qquad q_{v,j} \coloneqq \Pr{v_{j+1} \notin N_{j}[v]} = 1-\dfrac{d_{j}(v)+1}{|V_{j}|}.
	\end{align*}
	As $\Pr{v_i = u} = |V_{i-1}|^{-1}$ for all $u \in V_{i-1}$ and the choice of $v_i$ is independent from the $\sigma$-algebra $\F_{i-1}$, it follows from \eqref{eq:increment} that
	\begin{align}\label{eq:increment-2}
		\Ex{\Delta (X_{v,i}^{{\boldsymbol{\oplus}}}) 1_{\{\sigma_v > i\}} | \F_{i-1} } = -\dfrac{M_{v,i-1} 1_{ \{ v \in V_{i-1} \}}}{|V_{i-1}|} + \big ( - \Delta \td_i \, {\boldsymbol{\oplus}} \,\Delta (f_i \td_i) \big )q_{v,i-1} 1_{ \{ v \in V_{i-1} \}}
	\end{align} 
	and that 
	\begin{align}\label{eq:increment-2-abs}
		\Ex{ | \Delta( X_{v,i}^{{\boldsymbol{\oplus}}}) 1_{\{\sigma_v > i\}} | \big | \F_{i-1} }  \le \dfrac{M_{v,i-1} 1_{ \{ v \in V_{i-1} \}}}{|V_{i-1}|} + | \Delta \td_i |  + | \Delta (f_i \td_i)|.
	\end{align} 

   Note that $\Delta \td_{i} = -p \td_{i-1}$, as $\td_i = (1-p)^i pn$ (by~\eqref{eq:expected-trajectory-degree}), and
   \begin{align}\label{eq:increment-fidi}
	\Delta (f_i \td_i)  = f_{i}\Delta \td_i + \td_{i-1} \Delta f_i = ((1-p)f_{i} - f_{i-1}) \td_{i-1} = 16f_{i-1}p\td_{i-1},
	\end{align}
	as $ f_i = \big (\frac{1+16p}{1-p} \big )^i f_0$ (see~\eqref{eq:error-function}).
	If the event $\{ \tau \ge i, v \in V_{i-1} \}$ occurs, then it follows from~\eqref{eq:bound-f},~\eqref{eq:initial-parameters-2} and~\eqref{eq:for-closed-neighbourhood} that
   \begin{align}\label{eq:M2}
	q_{v,i-1} = 1 - \dfrac{(1 \pm f_{i-1})\td_{i-1}+1}{(1\pm f_{i-1})(1-p)^{i-1}n} 
	=  1-p\pm 3f_{i-1}p.
   \end{align}

	Now, note that if the event $\{ \tau \ge i, \, v \in V_{i-1} \}$ occurs, then by~\eqref{eq:increment-fidi} and~\eqref{eq:M2} we obtain
	\begin{align}\label{eq:increment-det}
		\big ( - \Delta \td_i \, {\boldsymbol{\oplus}} \,\Delta (f_i \td_i) \big )q_{v,i-1} 
		& = (1 \, {\boldsymbol{\oplus}} \, 16f_{i-1})(1 - p \pm 3f_{i-1}p)p\td_{i-1}\nonumber \\
		& = (1+({\boldsymbol{\oplus}} 16 \pm 1)f_{i-1})p \td_{i-1}.
	\end{align}
	Above, we use that $p \ll f_0 < f_{i-1}$.

	Next, we estimate $M_{v,i-1}$. 
	Observe that $M_{v,i-1}$ equals the number of edges in $G_{i-1}$ which have one endpoint in $N_{i-1}(v)$ and another in $V_{i-1} \setminus N_{i-1}[v]$.
	By double-counting, we obtain
	\begin{align}\label{eq:M_1}
		M_{v,i-1} = \sum \limits_{u \in V_{i-1} \setminus N_{i-1}[v]} d_{i-1}(u,v) = \sum \limits_{u \in N_{i-1}(v)} d_{i-1}(u) - d_{i-1}(v) - 2e_{G_{i-1}}(N_{i-1}(v)),
	\end{align}
	where $e_{G_{i-1}}(N_{i-1}(v))$ denotes the number of edges induced by $G_{i-1}$ in the set $N_{i-1}(v)$.
	If the event $\{ \tau \ge i \}$ occurs, then we can replace $d_{i-1}(u)$ by $(1 \pm f_{i-1})\td_{i-1}$ for all $u \in N_{i-1}[v]$ in~\eqref{eq:M_1}.
	By using that $2e(N_{i-1}(v))\le d_{i-1}(v)\Delta_2$ by~\ref{item:P3}, we obtain that if the event $\{ \tau \ge i, \, v \in V_{i-1} \}$ occurs, then it follows from~\eqref{eq:bound-f}--\eqref{eq:for-closed-neighbourhood} and~\eqref{eq:M_1} that
	\begin{align}
		\dfrac{M_{v,i-1}}{|V_{i-1}|} & = \dfrac{(1\pm f_{i-1})\td_{i-1}}{|V_{i-1}|}\left ( (1\pm f_{i-1})\td_{i-1} - 1 \pm \Delta_2 \right )  \nonumber \\
		& = (1\pm f_{i-1}) \left (1 \pm \frac{3}{2}f_{i-1} \right ) p \left ( \left (1\pm \frac{4}{3}f_{i-1} \right )\td_{i-1} + o(f_{i-1}\td_{i-1}) \right )   \nonumber \\
		& =  (1\pm 4 f_{i-1}) p\td_{i-1}. \label{eq:bound-M1-new}
   \end{align}

	Next, we combine~\eqref{eq:increment-2},~\eqref{eq:increment-det} and~\eqref{eq:bound-M1-new}, and obtain 
	\begin{align*}
		\Ex{\Delta X_{v,i}^{{\boldsymbol{\oplus}}} | \F_{i-1} } & = 
		\Ex{\Delta (X_{v,i}^{{\boldsymbol{\oplus}}}) 1_{ \{ \sigma_v > i\}}| \F_{i-1} } 1_{\{ \tau \ge i \}}  \\
		& = ({\boldsymbol{\oplus}} 16 \pm 5)f_{i-1}p \td_{i-1} 1_{\{ \tau \ge i,\, v \in V_{i-1} \}}.
	\end{align*} 
	
	Thus,
	$\Ex{\Delta X_{v,i}^{-} | \F_{i-1} } \le 0$ and $\Ex{\Delta X_{v,i}^{+} | \F_{i-1} } \ge 0$.
	Similarly, by combining~\eqref{eq:increment-2-abs},~\eqref{eq:increment-fidi} and~\eqref{eq:bound-M1-new}, and using that $f_{i-1} = o(1)$ and $\Delta \td_{i} = -p \td_{i-1}$, we obtain
	\begin{align*}
		\Ex{|\Delta X_{v,i}^{ {\boldsymbol{\oplus}} }| \big | \F_{i-1} }\le 3p\td_{i-1}.
	\end{align*}

	Now it remains to bound $|\Delta X_{v,i}^{ {\boldsymbol{\oplus}} }|$ from above.
	If the event $\{ \sigma_v \le i\}$ occurs, then $|\Delta X_{v,i}^{ {\boldsymbol{\oplus}} }|=0$.
	Otherwise, $|\Delta X_{v,i}^{ {\boldsymbol{\oplus}} }|\le \max_{u \in V_{i-1}} d_{0}(u,v) + |\Delta \td_i| + |\Delta (f_i \td_i)|$.
	By property~\ref{item:P3} combined with~\eqref{eq:increment-fidi} and the fact that $|\Delta (f_i \td_i)| \ll |\Delta \td_i| = |p\td_{i-1}| \le |p\td_{0}| = p^2n$, we have 
	$\max_i |\Delta X_{v,i}^{ \boldsymbol{\oplus} }| \le 6p^2n + 2^7\log n$, which completes the proof.
\end{proof}

In order to prove Theorem~\ref{thm:main-dem}, we shall use Freedman's inequality.

\begin{lemma}[Freedman's inequality]\label{lem:freedman}
Let $(X_i, \mathcal{F}_i)_{i=0}^m$ be a supermartingale in a probability space $(\Omega,\mathcal{F},\mathbb{P})$,
where $\mathcal{F}_i \se \mathcal{F}$ and $X_i$ is a random variable in $(\Omega,\mathcal{F}_i)$ for all $i$.
Let $R \in \mathbb{R}$ be such that $\max_i | X_i - X_{i-1}| \le R$. Let $V(m)$ be the quadratic variation of the process until step $m$, that is,
\[
V(m) := \sum_{i=1}^m \Ex{(X_i - X_{i-1})^2 \mid \mathcal{F}_{i-1}}.
\]
Then, for every $t, s > 0$, we have
\[
\Pr{X_m - X_0 \geq t \text{ and }  V(m) \leq s } \le \exp\left(-\frac{t^2}{2(s + Rt)}\right).
\]
\end{lemma}

Now we are ready to prove Theorem~\ref{thm:main-dem}.

\begin{proof}[Proof of Theorem~\ref{thm:main-dem}]
	For every $i \in [k-1]$, the event $\{\tau = i\}$ occurs if and only if there exists a vertex $v \in V_i$ such that $d_i(v) \neq (1 \pm f_i)\td_i$. 
    Thus, it suffices to show that $\Pr{\tau < k} \le n^{-2^{-11}\log pn}$.

	Let $i \in [k]$ and $v \in V_i$. 
	We have $d_i(v) > (1 + f_i)\td_i$ or $d_i(v) < (1 - f_i)\td_i$ if and only if $X_{v,i}^{-} > 0$ or $X_{v,i}^{+}< 0$, respectively.
	Then, we can write
	\begin{align*}
		\Pr{ \tau = i } \le  \sum_{v \in V(G)} \Pr{ \{\tau \ge i \} \cap \{ X_{v,i}^{-} > 0 \}} + \Pr{ \{\tau \ge i \} \cap \{ X_{v,i}^{+} < 0 \}}.
	\end{align*}
	By Property~\ref{item:P2}, we have $d_0(v) =  \big (1 \pm \frac{f_0}{2} \big )\td_0$, which implies that $X_{v,0}^{-} \le - \td_0 f_0/2$ and $X_{v,0}^{+} \ge \td_0 f_0/2$
	for all $v \in V_0$.
	In particular, if either $\{ X_{v,i}^{-} > 0 \}$ or $\{ X_{v,i}^{+} < 0 \}$ occur, then we must have either
	\begin{align*}
		X_{v,i}^{-}-X_{v,0}^{-} \ge \dfrac{\td_0 f_0}{2} \qquad \text{or} \qquad -X_{v,i}^{+}+X_{v,0}^{+} \ge \dfrac{\td_0 f_0}{2},
	\end{align*}
	respectively.
	By Lemma~\ref{lem:supermartingale}, the random processes $(X^{-}_{v,j})_{j=0}^i$ and $(-X^{+}_{v,j})_{j=0}^i$ are supermartingales.
	Thus, we can bound the probability of occurence of these events via Freedman's inequality (Lemma~\ref{lem:freedman}).

	It follows from Lemma~\ref{lem:supermartingale} (ii) and (iii)  that
	\begin{align*}
		\sum_{j=1}^i \Ex{(\Delta X_{v,j}^{\boldsymbol{\oplus}})^2 \mid \mathcal{F}_{j-1}} 
		&\le \max_j |\Delta X_{v,j}^{\oplus}| \sum_{j=1}^i \Ex{|\Delta X_{v,j}^{\boldsymbol{\oplus}}| \mid \mathcal{F}_{j-1}} \\
		& \le 2^9(p^3n+p\log n) \sum \limits_{j=1}^i \td_{j-1} \\
		& = 2^9(p^4n^2+p^2n\log n) \sum \limits_{j=1}^i (1-p)^{j-1} \\
		& \le 2^9  (p^3n^2+pn\log n).
	\end{align*}

	By Freedman's inequality, and using that $f_0 = o(1)$ and $\td_0 = pn$, we have 
	\begin{align}\label{eq:freedman-appl-1}
		\Pr{ X_{v,i}^{-}-X_{v,0}^{-} \ge \dfrac{f_0\td_0}{2}} & \le 
		\exp\left(-\frac{(f_0\td_0)^2}{2^{12}(p^2n+\log n)(pn + f_0\td_0)}\right) \nonumber \\
		& \le \exp\left(-\frac{f_0^2pn}{2^{13}(p^2n+\log n)}\right) \nonumber \\
		& \le \exp\left(-\frac{\log^2 n \log pn + p^2n \log^2 n }{2^{9}(p^2n+\log n)}\right).
	\end{align}
	In the last inequality, we use the definition of $f_0$ (see~\eqref{eq:error-function}).
	Let $E$ denote the absolute value of the exponent in the last term. 
	If $p^2n \ge \log n$, then $E \ge 2^{-9}p^2n\log^2n(p^2n+\log n)^{-1} \ge 2^{-10}\log^2n$, and hence
	the right-hand side of~\eqref{eq:freedman-appl-1} is at most $\exp(-2^{-10}\log^2 n)$; if $p^2n < \log n$, then $E \ge 2^{-9}\log^2 n \log pn(p^2n+\log n)^{-1} \ge 2^{-10}\log n \log pn$, and hence
	the right-hand side of~\eqref{eq:freedman-appl-1} is at most $\exp(-2^{-10}\log n \log pn)$.
	Therefore, in both cases we have
	\begin{align*}
		\Pr{X_{v,i}^{-}-X_{v,0}^{-} \ge \dfrac{f_0\td_0}{2}} \le n^{-2^{-10}\log pn} \quad \text{and} \quad 
		\Pr{-X_{v,i}^{+}+X_{v,0}^{+} \ge \dfrac{f_0\td_0}{2}} \le n^{-2^{-10}\log pn}.
	\end{align*}
	Summing over all $v \in V_0$ and $i \in [k-1]$, we obtain that $\Pr{\tau < k} \le n^{-2^{-11}\log pn}.$
\end{proof}

\section{Proof of Theorem~\ref{thm:prob-non-edge}}\label{sec:prob-cover} 

In this section we prove Theorem~\ref{thm:prob-non-edge}. 
Recall the random greedy independent set process described in Section \ref{sec:proofs} and its analysis in Section~\ref{sec:random-greedy}.
We denote $G_0 = G$, $V_0 = V(G_0)$, and $v_1, \ldots, v_i$ the vertices selected in the first $i$ steps of the process. 
For $i \ge 1$, we set $V_i = N^c(\{v_1,\ldots,v_{i}\})$ to be the common non-neighbourhood of $v_1, \ldots, v_i$.
As in Section~\ref{sec:random-greedy}, we denote by $\tau$ the minimum between $k =  \eps 2^{-10}p^{-1} \log pn$ and the first time $i$ that $d_i(v) \neq (1 \pm f_i) \td_i $ for some $v \in V_i$.

Conditioned on the event $\{\tau \ge i\}$, our next lemma yields the probability that a specified vertex is chosen at step $i$ given that it is in $G_{i-1}$, as well as the probability that a vertex or non-edge of $G_{i-1}$ is in $G_{i}$.

\begin{lemma}\label{lemma:prob-vertex-non-edge}
	Let $\eps \in (0,2^{-10})$ and $(\log n)^{2+\eps} n^{-1} \le p < n^{-\eps}$. 
	Let $G$ be an $n$-vertex $(\eps,p)$-typical graph and $k = \eps 2^{-10}p^{-1} \log pn $.
	Let $s \in [k]$ and let $J \se V(G)$ be an independent set in $G$ with $|J| \in [2]$.
	If $v \in J$ and $\E \se \{ J \se V_{s-1}, \, \tau \ge s \}$, with $\Pr{\E}>0$, then
	\begin{enumerate}[label=\normalfont(\roman*)]
		\item $\Pr{ v_s = v \big | \E } = (1\pm 2f_{s-1})(1-p)^{-s+1} n^{-1}$, and \vspace*{2pt}
		\item $\Pr{ J \se V_{s} \big | \E } = 1- |J|\big (1 \pm 35f_{s-1}/12 \big ) p.$
	\end{enumerate}
\end{lemma}

\begin{proof}
	Let $H \se G$ be an induced subgraph with 
	$\{G_{s-1} = H\} \se \E$ and $\Pr{G_{s-1} = H} > 0$.
	Observe that $\E$ is the disjoint union of such events for different choices of $H$.
	Note that assuming $\{G_{s-1} = H\}$ implies $v \in V(G_{s-1})$.
	In addition, $v_s$ is chosen uniformly from $V(G_{s-1})$.
	Hence, $\Pr{ v_s = v \big | G_{s-1} = H }  = v(H)^{-1}$.
	Note that, by the choice of $H$ and~\eqref{eq:initial-parameters-2}, we have $v(H) = (1\pm f_{s-1})(1-p)^{s-1} n$.
	By combining this with~\eqref{eq:bound-f}, we obtain 
	\begin{align}\label{eq:item-i-cond-H}
		\Pr{ v_s = v \big |  G_{s-1} = H} 
		& = \dfrac{1 }{v(H)}  = \dfrac{1 \pm 2f_{s-1}}{(1-p)^{s-1} n}.
	\end{align}

	By the choice of $H$, we have $d_H(u) = (1\pm f_{s-1})\td_{s-1}$ for all $u \in V(H)$.
	Thus, it follows from~\ref{item:P3} and the estimates~\eqref{eq:bound-f},~\eqref{eq:bound-Delta-2} and~\eqref{eq:for-closed-neighbourhood} that
	\begin{align}\label{eq:item-ii-cond-H}
		\Pr{ J \se V_{s} \big | G_{s-1} = H } = 1 - \dfrac{\big | \cup_{u \in J} N_H[u] \big |}{v(H)} 
		& = 1 - \dfrac{|J|\big ( (1\pm f_{s-1})\td_{s-1} + 1 \big ) \pm \Delta_2 }{v(H)} \nonumber \\
		& =  1- |J|\Big (1 \pm \big (17/6 +o(1)\big )f_{s-1} \Big ) p \nonumber\\
		& =  1- |J|\Big (1 \pm 35f_{s-1}/12 \Big ) p.
	\end{align}
	In the last inequality, we use that $(1\pm f_{s-1})\td_{s-1} + 1 \big ) \pm \Delta_2 = \Big ( 1 \pm \big (4/3+o(1) \big )f_i \Big )$
	and that $v(H)^{-1} = (1 \pm 3f_{s-1}/2)(1-p)^{-(s-1)}n^{-1}$.
	Items (i) and (ii) follow from~\eqref{eq:item-i-cond-H} and~\eqref{eq:item-ii-cond-H}, respectively, by summing over all induced subgraphs $H \se G$ such that 
	$\{G_{s-1} = H\} \se \E$ and $\Pr{G_{s-1} = H} > 0$.
\end{proof}

For $u,v \in V_0$, $0 \le i < j \le k$ and $0 \le t \le j$, set
\begin{align*}
	\E_t^{i,j}(u,v) \coloneqq \begin{cases*} 
		\big \{ u,v \in V_t, \, \tau > t \big \} & \text{ if } $0 \le t < i;$ \vspace{0.5pt}\\ 
		\big \{ v_i = u, \, v \in V_t, \, \tau > t \big \} & \text{ if } $i \le t < j;$ \vspace{0.5pt}\\
		\big \{ v_i = u, \, v_j = v, \, \tau > j \big \} & \text{ if } $t = j.$
	\end{cases*}
\end{align*}
Our next lemma is a crucial step towards proving Theorem~\ref{thm:prob-non-edge}, as it yields the asymptotic behaviour of $\Pr{\E_t^{i,j}(u,v) \big | \E_{t-1}^{i,j}(u,v)}$ for all $t \in [j]$.

\begin{lemma}\label{lem:prob-non-edge}
	Let $\eps \in (0,2^{-10})$ and $(\log n)^{2+\eps} n^{-1} \le p < n^{-\eps}$. 
	Let $G$ be an $n$-vertex $(\eps,p)$-typical graph and $k = \eps 2^{-10}p^{-1} \log pn $.
	For all distinct $u,v \in V_{0}$ such that $uv \notin E(G_{0})$ and $0 \le i < j \le k$, we have
	\begin{align*}
		\Pr{\E_t^{i,j}(u,v) \big | \E_{t-1}^{i,j}(u,v)} = 
		\begin{cases*}
			1 - (1 + \mathbbm{1}_{\{t < i\}})( 1\pm 3f_{t-1} )p & \text{ if } $t \in [j-1] \setminus \{i\};$ \vspace{4pt}\\ (1 \pm 3f_{t-1})(1-p)^{-t+1} n^{-1} & \text{ if } $t \in \{i,j\}.$
		\end{cases*}
	\end{align*}
\end{lemma}

\begin{proof}
	Fix $i,j$ with $0 \le i < j \le k$ and distinct $u,v \in V_{0}$ such that $uv \notin E(G_{0})$.
	For simplicity, write $\E_t = \E_t^{i,j}(u,v)$ for all $t$.

	Let $t \in [i-1]$. It follows from the bound $\Pr{\tau = t} \le n^{-2^{-11}\log pn}$ given by Theorem~\ref{thm:main-dem} and Lemma~\ref{lemma:prob-vertex-non-edge}(ii) with $s=t$, $J = \{u,v\}$ and $\E = \E_{t-1}$ that
	\begin{align}\label{eq:first-cond-Et}
		\Pr{\E_{t}|\E_{t-1}} & = \Pr{u,v \in V_{t} | \E_{t-1}} \pm \Pr{\tau = t | \E_{t-1}}  \nonumber \\
		& = 1- 2\big (1 \pm 35f_{t-1}/12 \big ) p \pm \dfrac{n^{-2^{-11}\log pn}}{\Pr{\E_{t-1}}}.
	\end{align}
	We shall use~\eqref{eq:first-cond-Et} to prove by induction that $ \Pr{\E_{t}|\E_{t-1}} = 1 - 2(1\pm 3f_{t-1})p$ for all $t \in [i-1]$.
	By~\ref{item:P2} we have $\Pr{\E_0} = \Pr{\tau \ge 1} =1$.
	As $n^{-2^{-11}\log pn} \ll n^{-3/2} \le f_{0}p$ by~\eqref{eq:asymptotics} and $pn \gg 1$, we obtain from~\eqref{eq:first-cond-Et} that $\Pr{\E_1|\E_0} = 1 - 2(1\pm 3f_{0})p$.
 	This settles the base case.

 	Suppose that $2 \le t \le i-1$ and that $\Pr{\E_s|\E_{s-1}} = 1 - 2(1\pm 3f_{s-1})p \ge 1-3p$ for all $s \in [t-1]$.
	As $\E_s \se \E_{s-1}$ and $\Pr{\E_0}=1$, it follows from the induction hypothesis that $\Pr{\E_s} \ge (1-3p)^s$ for all $s \in [t-1]$.
	By combining~\eqref{eq:asymptotics} and~\eqref{eq:first-cond-Et}, we obtain
	\begin{align}\label{eq:lower-bound-Et}
		\Pr{\E_{t}|\E_{t-1}} & = 1- 2\big (1 \pm 35f_{t-1}/12 \big ) p \pm \dfrac{n^{-2^{-11}\log pn}}{(1-3p)^{t-1}} \nonumber\\
		& = 1- 2\big (1 \pm 35f_{t-1}/12 \big ) p \pm n^{-2^{-12}\log pn} \nonumber\\
		& = 1 - 2(1\pm 3f_{t-1})p \ge 1-3p.
	\end{align}
	This completes the inductive argument.

	Next, we show that $\Pr{\E_{i} \big | \E_{i-1}}\approx (1-p)^{-i+1}n^{-1}$.
	Indeed, we have
	\begin{align*}
		\Pr{\E_{i}} & = \Pr{v_i = u, \, \{ u, v \} \se V_{i-1}, \, \tau \ge i  } \pm \Pr{ \tau = i}\\
		& = \Pr{v_i = u, \, \E_{i-1} } \pm n^{-2^{-11}\log pn}\\
		& = \dfrac{1 \pm 2f_{i-1}}{(1-p)^{i-1} n}\Pr{\E_{i-1}} \pm n^{-2^{-11}\log pn}.
	\end{align*}
	In the second equality, we use Theorem~\ref{thm:main-dem}; and in the third inequality we use~\eqref{eq:bound-f} and Lemma~\ref{lemma:prob-vertex-non-edge}(i) with $s=i$, $J =\{u,v\}$ and $\E = \E_{i-1}$.
	It follows from~\eqref{eq:asymptotics} and~\eqref{eq:lower-bound-Et} that
	\begin{align*}
		 \dfrac{f_{i-1}\Pr{\E_{i-1}}}{(1-p)^{i-1} n} \ge \dfrac{f_0(1-3p)^{i-1}}{n} \ge n^{-2^{-6}\eps-3/2}\gg n^{-2^{-11}\log pn},
	\end{align*}
	and hence, again by~\eqref{eq:asymptotics} and~\eqref{eq:lower-bound-Et},
	\begin{align}\label{eq:prob-Aii}
		\Pr{\E_{i}} = \dfrac{1 \pm 3f_{i-1}}{(1-p)^{i-1} n}\Pr{\E_{i-1}} \ge \dfrac{1}{2n}\Pr{\E_{i-1}} \ge \dfrac{(1-3p)^{i-1}}{2n} \gg n^{-2^{-12}\log pn}.
	\end{align}

	Let $t \in \{ i+1,\ldots,j-1\}$.
	Analogously to~\eqref{eq:first-cond-Et}, it follows from Theorem~\ref{thm:main-dem} and Lemma~\ref{lemma:prob-vertex-non-edge}(ii) with $s=t$, $J = \{v\}$ and $\E = \E_{t-1}$ that
	\begin{align}\label{eq:first-cond-Et-2}
		\Pr{\E_{t}|\E_{t-1}}  & = \Pr{v \in V_t |\E_{t-1}} \pm \Pr{\tau = t |\E_{t-1}}  \nonumber \\
		& = 1- \big (1 \pm 35f_{t-1}/12 \big ) p \pm \dfrac{n^{-2^{-11}\log pn}}{\Pr{\E_{t-1}}}.
	\end{align}
	Next, we show that $\Pr{\E_{t}|\E_{t-1}} = 1 - (1\pm 3f_{t-1})p$ for $t \in \{ i+1,\ldots, j-1\}$, again by induction.
	By combining~\eqref{eq:prob-Aii} with~\eqref{eq:first-cond-Et-2} we obtain $\Pr{\E_{i+1}|\E_{i}} = 1 - ( 1\pm 3f_{i} )p.$
	This establishes the base case of the induction.

	Suppose that $t \in \{ i+1,\ldots,j-1\}$ and that $\Pr{\E_s|\E_{s-1}} = 1 - (1\pm 3f_{s-1})p \ge 1-2p$ for all $s \in \{i+1,\ldots,t-1\}$.
	As $\E_{s} \se\E_{s-1}$ and $f_{s-1} = o(1)$, it follows from~\eqref{eq:asymptotics},~\eqref{eq:prob-Aii} and the induction hypothesis that
	\begin{align*}
		\Pr{\E_{t-1}} = \Pr{\E_{i}}  \prod \limits_{s=i+1}^{t-1} \dfrac{\Pr{\E_s}}{\Pr{\E_{s-1}}}  = \Pr{\E_{i}}  \prod \limits_{s=i+1}^{t-1} \Pr{\E_s|\E_{s-1}} \ge \dfrac{(1-2p)^k}{2n} \ge n^{-2^{-12}\log pn}.
	\end{align*}
	We complete the inductive argument by combining these bounds with~\eqref{eq:asymptotics} and~\eqref{eq:first-cond-Et-2}.

	Finally, it remains to show that $\Pr{\E_{j}|\E_{j-1}} = (1 \pm 3f_{j-1})(1-p)^{-j+1} n^{-1}$. For this, note that, by Theorem~\ref{thm:main-dem} and Lemma~\ref{lemma:prob-vertex-non-edge}(i) with $s=j$, $J = \{u,v\}$ and $\E = \E_{j-1}$, we have
	\begin{align}\label{eq:Ej1}
		\Pr{\E_{j}} = \Pr{v_j = v, \, v_i = u, \, \tau \ge j+1  } & =  \Pr{v_j = v, \, \E_{j-1} }
		\pm \Pr{ \tau = j} \nonumber \\
		& = \dfrac{1 \pm 2f_{j-1}}{(1-p)^{j-1} n}\Pr{\E_{j-1}} \pm n^{-2^{-11}\log pn}. 
	\end{align}
	As $\Pr{\E_s|\E_{s-1}} \ge 1-3p$ for all $s \in [j-1] \setminus \{i\}$ and $\Pr{\E_i|\E_{i-1}} \ge (2n)^{-1}$ by~\eqref{eq:prob-Aii}, we obtain
	\begin{align}\label{eq:Ej2}
		\Pr{\E_{j-1}} =  \Pr{\E_{0}}  \prod \limits_{s=1}^{j-1} \dfrac{\Pr{\E_s}}{\Pr{\E_{s-1}}} = \prod \limits_{s=1}^{j-1} \Pr{\E_s|\E_{s-1}}  \ge  \dfrac{(1-3p)^{j-2}}{2n}.
	\end{align}

	From~\eqref{eq:error-function},~\eqref{eq:asymptotics} and~\eqref{eq:Ej2} we obtain
	\begin{align*}
		 \dfrac{f_{j-1}\Pr{\E_{j-1}}}{(1-p)^{j-1} n} \ge \dfrac{f_0(1-3p)^{j-2}}{2n^2} \gg  n^{-2^{-11}\log pn},
	\end{align*}
	and hence, by~\eqref{eq:Ej1} we have
	\begin{align*}
		\Pr{\E_{j}} = \dfrac{1 \pm 3f_{j-1}}{(1-p)^{j-1} n}\Pr{\E_{j-1}}.
	\end{align*}
	This concludes the proof.
\end{proof}

Now we are ready to prove Theorem~\ref{thm:prob-non-edge}.

\begin{proof}[Proof of Theorem~\ref{thm:prob-non-edge}]
	Let $G$ be as in the statement and let $u,v \in V(G)$ be distinct vertices such that $uv \notin E(G)$.
	Let $v_i$ be the vertex selected in the $i$th step of the random greedy independent set process.
	It suffices to show that, for every $i,j$ with $0 \le i < j \le k$, we have $ \Pr{v_i = u, v_j = v} = (1+o(1))n^{-2}$.

	From now on, we fix $i,j$ with $0 \le i < j \le k$.
	For simplicity, we write $\E_t = \E_t^{i,j}(u,v)$ for all $t$.
	Note that
	\begin{align*}
		\Pr{v_i = u, \, v_j = v, \, \tau > j} = \Pr{\E_{j}} =  \Pr{\E_{0}}  \prod \limits_{t=1}^{j} \dfrac{\Pr{\E_t}}{\Pr{\E_{t-1}}} = \prod \limits_{t=1}^{j} \Pr{\E_t|\E_{t-1}}.
	\end{align*}
	In the last inequality above, we use that $\Pr{\E_0} = \Pr{\tau \ge 1} =1$ (by~\ref{item:P2}) and that $\E_t \se \E_{t-1}$ for all $t \in [j]$. 
	Together with Lemma~\ref{lem:prob-non-edge}, this implies that
	\begin{align*}
		\Pr{v_i = u, v_j = v, \tau > j} & = \dfrac{1+o(1)}{(1-p)^{i+j-2} n^2}  \cdot \prod_{t=1}^{i-1} \big ( 1 - 2(1\pm 3f_{t-1})p \big ) \prod_{t=i+1}^{j-1} \big ( 1 - ( 1\pm 3f_{t-1} )p \big ) \\
		& = \dfrac{1+o(1)}{(1-p)^{i+j-2} n^2}  \cdot \big ( 1 - 2(1\pm 3f_{k})p \big )^{i-1} \big ( 1 - ( 1\pm 3f_{k} )p \big )^{j-i-1}.
	\end{align*}
	It is routine to verify that for every $x \in (0,1/2)$, we have $1-x = e^{-x \pm 2x^2}$.
	Consequently, the right-hand side of the display above equals 
	\begin{align*}
		(1+o(1))n^{-2} \exp \Big ( p(i+j) - 2(1\pm 3f_k)pi -(1\pm 3f_k)p(j-i) \pm O(p+p^2(i+j)) \Big ),
	\end{align*}
	and hence
	\begin{align}\label{eq:prob-A-6}
		\Pr{v_i = u, v_j = v, \tau > j} & = (1+o(1))n^{-2} \exp \left ( \pm O(p+(p^2+f_kp)(i+j)) \right )\nonumber\\
		& = (1+o(1))n^{-2} \exp \left ( \pm O(f_kpk) \right ).
	\end{align}
	In the last equality, we simply use that $i + j = O(k)$ and that $p^2 k = O(p\log pn) = o(1)$.

	Now we claim that $f_kpk = o(1)$.
	Indeed, it follows from~\eqref{eq:bound-f} and $k = \eps 2^{-10}p^{-1}\log pn$ that 
	\begin{align*}
		f_kpk \le (pn)^{2^{-5}\eps} f_0 \log pn \le h \coloneqq 4(pn)^{2^{-5}\eps}  \log n \left (\dfrac{\log pn}{pn} + p \right )^{1/2} \log pn.
	\end{align*}
	Thus, it suffices to show that $h^2\to 0$ as $n\to \infty$.
	If $(pn)^{-1}\log pn \le p$, then 
	\begin{align*}
		h^2 = O\Big ((\log n)^4 n^{2^{-4}\eps}p^{1+2^{-4}\eps} \Big ) = O \big (n^{2^{-3}\eps}p \big ) = o(1),
	\end{align*}
	because $p \le n^{-\eps}$ by assumption.
	If $(pn)^{-1}\log pn \ge p$, then
	\begin{align*}
		h^2 = O \big ((\log n)^2 (\log pn)^{3} (pn)^{-1+ 2^{-4} \eps} \big ).
	\end{align*}
	Observe that $h^2\to 0$ if and only if $((\log n)^2 (\log pn)^3)^{\frac{1}{1-2^{-4}\eps}} \ll pn$. But now note that
	\begin{align*}
		 \dfrac{((\log n)^2 (\log pn)^{3})^{\frac{1}{1-2^{-4}\eps}} }{n} \le \dfrac{((\log n)^2 (\log pn)^{3})^{1+\eps/4} }{n} \ll  p.
	\end{align*}
	In the first inequality, we use that $(1-2^{-4}\eps )^{-1} \le 1 + \eps/4$ for all $\eps \in (0,1/2)$.
	In the second inequality, we use $(\log n)^{2+\eps}n^{-1}\le p$ by assumption.
	It follows from~\eqref{eq:prob-A-6} that $\Pr{v_i = u, v_j = v, \tau > j} = (1+o(1))n^{-2}$.
	This together with Theorem~\ref{thm:main-dem} completes the proof.
\end{proof}

\section{Proof of Lemma~\ref{lem:typical}}\label{sec:preliminaries}

In this section, we prove Lemma~\ref{lem:typical}. Our main tool is Chernoff's inequality.

\begin{lemma}[Chernoff's inequality]\label{lem:chernoff}
	Let $Y$ be a binomial random variable. For every $t \ge 0$, we have
	\begin{align*}
		\Pr{|Y - \Ex{Y}| \ge t} \le 2 \exp \left ( -\dfrac{t^2}{2\Ex{Y} + t} \right ).
	\end{align*}
\end{lemma}

\begin{proof}[Proof of Lemma~\ref{lem:typical}]
	Let $\eps$, $n$ and $p$  as in the statement.
To simplify notation, for every set $S \se V(G_{n,p})$ of size at most $\eps 2^{-10}p^{-1} \log pn$, we set
\begin{align*}
	 \mu_{|S|} \coloneqq (1-p)^{|S|} n \qquad \text{and} \qquad E_{|S|} \coloneqq \Ex{|N^c(S)|} = (n-|S|)(1-p)^{|S|}.
\end{align*}

We start by first showing that $E_i \ge 4i\log n$ for every $i \in \big [\eps 2^{-10}p^{-1} \log pn \big ]$. 
This will be useful later when we apply Chernoff's inequality to prove that~\ref{item:P1} holds with high probability.
Note that $pn \ge (\log n)^{2+\eps} \gg  (\log n)^{\frac{2}{1-2^{-9}\eps}}$. By raising both sides to the power of $1-2^{-9}\eps$ and using that $\log n \ge \log np$, we obtain $(pn)^{1-2^{-9}\eps} \gg \log n \log np$, and hence 
\begin{align*}
	(pn)^{-2^{-9}\eps}\dfrac{p}{\log pn}\gg \dfrac{\log n}{n}.
\end{align*}
As $1-x \ge e^{-2x}$ for all $x \in (0,1/2)$ and $i \to e^{-2pi}/i$ is decreasing, for every $i \in \big [\eps 2^{-10}p^{-1} \log pn \big ]$, we have
\begin{align*}
	\dfrac{(1-p)^i}{i} \ge \dfrac{e^{-2pi}}{i} \ge (pn)^{-2^{-9}\eps}\dfrac{2^{10}p}{\eps\log pn}\gg \dfrac{\log n}{n}.
\end{align*}
This implies that
\begin{align}\label{eq:bound-e-i}
	E_i = (1-p)^i (n-i) \ge (8\log n - (1-p)^i)i \ge 4i\log n
\end{align}
for every $i \in \big [\eps 2^{-10}p^{-1} \log pn \big ]$, which proves our claim.

For $i \in \big [\eps 2^{-10}p^{-1} \log pn \big ]$, set $t = t(i) = 4\sqrt{iE_i \log n}$.
Now we claim that $t^2 \ge 8i(\log n)\max(2E_i, t)$.
Indeed, we have $t^2 = 16iE_i\log n$ by choice of $t$.
If $2E_i \ge t$, then the statement is trivial. 
Otherwise, this is equivalent to $t \ge 8i\log n$, and~\eqref{eq:bound-e-i} proves this claim.
By Chernoff's inequality combined with the union bound, it follows that the probability that $ \big | | N^c(S) | - E_{|S|} \big | \ge t$ for some set $S$ of size $i \le \eps 2^{-10}p^{-1} \log pn$ is at most 
	\begin{align*}
		n^i \exp \left ( -\dfrac{t^2}{2E_i + t} \right ) \le n^i \exp \left ( -\dfrac{t^2}{2\max(2E_i, t)} \right ) \le n^{-3i}.
	\end{align*}
	Since $|S| \le \sqrt{|S|n}$ implies that $|S|(1-p)^{|S|} \le \sqrt{|S|n(1-p)^{|S|}} = \sqrt{|S|\mu_{|S|}}$, we conclude that with probability at least $1-\sum_i n^{-3i} = 1-O(n^{-3})$ for every set $S \se V(G_{n,p})$ of size at most $\eps 2^{-10}p^{-1} \log pn$, we have
	\begin{align*}
		|N^c(S)| &= E_{|S|} \pm 4\sqrt{|S|E_{|S|}\log n} \nonumber \\
		& = \mu_{|S|}-(1-p)^{|S|}|S| \pm 4\sqrt{|S|\mu_{|S|}\log n} \nonumber \\
		& = \mu_{|S|} \pm 5\sqrt{|S|\mu_{|S|}\log n} =  \mu_{|S|} \left ( 1 \pm 5 \sqrt{\dfrac{|S|\log n}{\mu_{|S|}}} \right ).
	\end{align*}
	
	Thus,~\ref{item:P1} holds with probability at least $1-O(n^{-3})$, provided that $25\mu_{s}^{-1}s\log n \le f_{s}^2$ holds for every $s \in [\eps 2^{-10}p^{-1} \log pn]$.
	Indeed, note that
		\begin{align*}
			\dfrac{25s\log n}{n}\le \dfrac{\log n\log pn}{pn} \le f_0^2 = 16 \log^2 n \left ( \dfrac{\log pn}{pn} + p \right ).
		\end{align*}
		By multiplying both sides by $(1-p)^{-s}$, we obtain 
		\begin{align*}
			\dfrac{25s\log n}{\mu_{s}} = \dfrac{25s\log n}{n(1-p)^{s}}  \le \dfrac{f_0^2}{(1-p)^{s}} \le f_0^2 \dfrac{(1+16p)^{2s}}{(1-p)^{2s}} = f_{s}^2,
		\end{align*}
		and this proves~\ref{item:P1}.

	Finally, it remains to show that~\ref{item:P2} and~\ref{item:P3} hold with probability at least $1-O(n^{-3})$.
	Again by Chernoff's inequality, we obtain
	\begin{align*}
		\Pr{|d(v)-pn|\ge \dfrac{f_0pn}{3}} \le 2 \exp \left ( -\dfrac{(f_0pn)^2}{9(2pn + f_0pn)} \right ) \le \exp \big ( -\Omega(\log^2 n) \big)
	\end{align*}
for every $v \in V(G_{n,p})$.
Here, we use that $f_0 \le 1 $ and that $\log^2 n \le f_0^2pn$.
By taking the union bound over all vertices in $V(G_{n,p})$, we obtain that~\ref{item:P2} holds with probability at least $1-O(n^{-3})$.

For~\ref{item:P3}, again by Chernoff's inequality, for every pair of distinct vertices $u,v \in V(G_{n,p})$, we obtain
\begin{align*}
	\Pr{\big | | N(u) \cap N(v)| - p^2n \big|\ge \dfrac{\Delta_2}{3}} \le 2 \exp \left ( -\dfrac{\Delta_2^2}{9(2p^2n + \Delta_2)} \right ) \le 2 \exp \left ( -\dfrac{\Delta_2}{18} \right ) = o(n^{-5}).
\end{align*}
Here, we use that $\Delta_2 \ge 4p^2n$ and $\Delta_2 \ge 2^7 \log n$.
By taking the union bound over all pairs of vertices in $V(G_{n,p})$, it follows that~\ref{item:P3} holds with probability at least $1-O(n^{-3})$.
This completes the proof.
\end{proof}

\end{document}